\documentclass[11pt]{amsart}
\usepackage{indentfirst,latexsym,bm}
\usepackage{amsfonts}
\usepackage{amssymb}
\usepackage{amsmath}
\usepackage{hyperref}
\usepackage{dsfont}
\usepackage{leftidx}
\usepackage{amsthm}
\usepackage{geometry}
\usepackage{amsmath,amscd}
\usepackage{tikz}
\usepackage[all,cmtip]{xy}

\usepackage{latexsym,amsmath,amssymb,amsfonts}  
\usepackage{color}
\usepackage{color,xcolor}
\usepackage{colordvi}
\usepackage{xspace}
      
\usepackage{graphicx}
\usepackage{pstricks}
\usepackage{pgf}
\usepackage{tikz}
\usepackage{xkeyval}
\usepackage{tikz}
\usetikzlibrary{mindmap,trees,shadows}

\usetikzlibrary{graphs}

\begin{document}
\thispagestyle{empty}

\title[Affineness on Noetherian graded rings, algebras and Hopf algebras]
{Affineness on Noetherian graded rings, algebras and Hopf algebras}

\author{
{Huan Jia $^{1,2}$}
and 
{Yinhuo Zhang $^{2,*}$}}

\address{$^{1}$ Department of Mathematics, Suqian Unviersity, Suqian City 223800, Jiangsu Province, China}

\address{$^{2}$ Department of Mathematics and Statistics,
University of Hasselt, Universitaire Campus, 3590 Diepenbeek, Belgium}

\address{Huan Jia $^{1,2}$}
\email{huan.jia@squ.edu.cn}

\address{Yinhuo Zhang $^{2,*}$}
\email{yinhuo.zhang@uhasselt.be}
 
\date{\today}
\maketitle

\newtheorem{theorem}{Theorem}[section]
\newtheorem{proposition}[theorem]{Proposition}
\newtheorem{lemma}[theorem]{Lemma}
\newtheorem{corollary}[theorem]{Corollary}
\theoremstyle{definition}
\newtheorem{definition}[theorem]{Definition}
\newtheorem{example}[theorem]{Example}
\newtheorem{remark}[theorem]{Remark}

\newcommand{\A}{\mathcal{A}}
\newcommand{\B}{\mathcal{B}}
\newcommand{\C}{\mathcal{C}}
\newcommand{\D}{\mathcal{D}}
\newcommand{\M}{\mathcal{M}}
\newcommand{\K}{\mathds{k}}
\newcommand{\T}{\mathcal{T}}
\newcommand{\Z}{\mathbb{Z}}

\newcommand{\Tt}{\overline{\mathcal{T}}}
\newcommand{\BN}{\mathcal{B}}
\newcommand{\N}{\mathds{N}}
\newcommand{\Pp}{\mathcal{P}}
\def\cL{\mathcal{L}}
\newcommand{\cI}{\mathcal{I}}
\newcommand{\cF}{\mathcal{F}}
\newcommand{\bN}{\mathbb{N}}
\newcommand{\X}{\langle X \rangle}
\newcommand{\Y}{\langle Y \rangle}
\newcommand{\Xo}{\langle X_{0} \rangle}
\newcommand{\Xge}{\langle X_{\ge 1} \rangle}
\newcommand{\Yo}{\langle Y_{0} \rangle}
\newcommand{\So}{\langle S_{0} \rangle}
\newcommand{\XI}{\langle X|I \rangle}
\newcommand{\al}{\alpha}
\newcommand{\cM}{\mathcal{M}}
\newcommand{\cN}{\mathcal{N}}
\newcommand{\cs}{\text{c}}

\newcommand\ad{\operatorname{ad}}
\newcommand\Ob{\operatorname{Ob}}
 \newcommand\Aut{\operatorname{Aut}}
 \newcommand\coker{\operatorname{coker}}
\newcommand\car{\operatorname{char}}
\newcommand\Der{\operatorname{Der}}
\newcommand\diag{\operatorname{diag}}
\newcommand\End{\operatorname{End}}
\newcommand\mult{\operatorname{mult}}
\newcommand\id{\operatorname{id}}
\newcommand\Char{\operatorname{char}}
\newcommand\gr{\operatorname{gr}}
\newcommand\GKdim{\operatorname{GKdim}}
\newcommand\GK{\operatorname{GK}}
\newcommand{\Hom}{\operatorname{Hom}}
\newcommand\ord{\operatorname{ord}}
\newcommand\rk{\operatorname{rk}}
\newcommand\Soc{\operatorname{soc}}
\newcommand\supp{\operatorname{supp}}
\newcommand\Rep{\operatorname{Rep}}
\newcommand\corad{\operatorname{corad}}
\newcommand\PBW{\operatorname{PBW}}
\newcommand{\h}{\operatorname{h}}
\newcommand{\Q}{\operatorname{Q}}
\newcommand\lex{\operatorname{lex}}
\newcommand\glex{\operatorname{glex}}
\newcommand\flex{\operatorname{flex}}
\newcommand\LS{\operatorname{LS}}
\newcommand\LW{\operatorname{LW}}
\newcommand\G{\operatorname{G}}
\newcommand\Sh{\operatorname{Sh}}
\newcommand\rank{\operatorname{rank}}
\newcommand{\st}{\operatorname{st}}
\newcommand{\BFdB}{\mathcal{B}_{\operatorname{FdB}}}
\newcommand{\BFdBnc}{\mathcal{B}_{\operatorname{FdB}}^{\operatorname{nc}}}
\newcommand{\HFdB}{\mathcal{H}_{\operatorname{FdB}}}
\newcommand{\HFdBnc}{\mathcal{H}_{\operatorname{FdB}}^{\operatorname{nc}}}
\newcommand\fd{\operatorname{fd}}

\newcommand{\ULa}{\underline{\mathbf{a}}}
\theoremstyle{plain}
\newcounter{maint}
\renewcommand{\themaint}{\Alph{maint}}
\newtheorem{mainthm}[maint]{Theorem}

\theoremstyle{plain}
\newtheorem*{proofthma}{Proof of Theorem A}
\newtheorem*{proofthmb}{Proof of Theorem B}

\parskip 3mm

\begin{abstract}
\bigskip 
In this note, we show that every Noetherian graded ring with an affine degree zero part is affine. As a result, a Noetherian graded Hopf algebra whose  degree zero component is a commutative or a cocommutative Hopf subalgebra is affine. 
Moreover, we show that the braided Hopf algebra of a Noetherian graded Hopf algebra is affine.
\vspace{2mm}

\end{abstract}

\section*{Introduction}
In 1997 Brown \cite{B1998} proposed to study ``Noetherian Hopf algebras" as an independent topic of infinite-dimensional Hopf algebras and raised several  questions and conjectures about Noetherian Hopf algebras,  see \cite{B2007,BG2014,G2013,BZ2020,A2023} for more details.
In 2003 Wu and Zhang proposed the following classical question, which is important for the classification and the structure study of Noetherian Hopf algebras:

\noindent
\textbf{Question  \cite{WZ2003}} 
\textit{Is every Noetherian Hopf k-algebra an affine $\K$-algebra? }

To  the authors' knowledge, there are  no counterexamples so far and the known results are relatively limited. When a Hopf algebra is commutative,  the Noetherian property is equivalent to the affine property. This was proved by Molnar in 1975: 
\begin{lemma}{\cite{M1975}}\label{lem:Molnar}
\textup{(a)} A commutative Hopf algebra is Noetherian if and only if it is an affine $\K$-algebra.\\
\textup{(b)}  A cocommutative Noetherian Hopf algebra is affine.
\end{lemma}
In the case of a PI Hopf algebra which is  close to a commutative Hopf algebra, we don't know if the two properties  are equivalent,  see \cite{BG2014,G2013}. 

In this note, we investigate the affine property of a Noetherian Hopf algebra in the graded case.  We first show in general that a Noetherian graded ring is affine if its degree zero part is affine, see Theorem \ref{thm:noetherian-graded-ring-affine}. As  a result, a Noetherian Hopf algebra (graded as an algebra) is affine if its degree zero part is a commutative or a cocommutative Hopf subalgebra, see Theorem \ref{thm:noetherian-graded-algebra-affine} and Corollary \ref{cor:noetherian-graded-Hopfalgebra-affine}. 
Since every Hopf algebra has a Hopf algebra filtration  and the associated graded Hopf algebra is a Radford biproduct \cite{R2012} of the degree zero part and a braided Hopf algebra, 
we show that the braided Hopf algebra of a Noetherian graded  Hopf algebra is affine, see Theorem \ref{thm:braided-part-affine}.

\section{Noetherian graded rings and graded algebras}

Let $\K$ be a field and $\mathbb{N}$ the set of natural numbers. The \textit{coradical} $\corad(C)$ of a coalgebra $C$ over $\K$ is the sum of all simple subcoalgebras of $C$. A ring is called \textit{affine} if it is finitely generated as a ring. 

Let $R=\oplus_{n\ge 0} R_{n}$ be an $\mathbb{N}$-graded ring and denote by $R_{+}:= \oplus_{n\ge 1} R_{n}$. First we show the following result for generators of a graded ring.

\begin{lemma}\label{lem:generators-noetherian-graded-ring}
Let $R=\oplus_{n\ge 0}R_{n}$ be a graded ring and $S_{0}$ a generating set of $R_{0}$. The following hold:
\begin{itemize}
  \item [(a)] Let $S_{+}$ be a subset of homogeneous elements of $R_{+}$. If $R_{+}=S_{+}R$ (resp. $R_{+}=RS_{+}$), then $R=\langle S_{0} \cup S_{+} \rangle$, that is, if $S_{+}$ generates $R_{+}$ as a right ideal (resp. left ideal) of $A$, then $S_{0} \cup S_{+}$ generates $R$ as a ring.
  \item [(b)] If $R$ is right or left noetherian, then there is a finite subset $S_{+}\subseteq R_{+}$ such that $R=\langle S_{0} \cup S_{+}\rangle$. 
\end{itemize}
\end{lemma}

\begin{proof}
Clearly, $R_{+}$ is a right ideal of $R$. To show that $S_{0} \cup S_{+}$ generates $R$ as a ring, it is enough to show that every homogeneous element $r\in R$ belongs to the subring $\langle S_{0} \cup S_{+}\rangle$ of $R$ generated by $S_{0} \cup S_{+}$. It holds for $r\in R_{0}$. Now let $r\in R_{n}$ with $n\ge 1$. Recall that $R_{+}=S_{+}R$. Thus,  there are homogeneous elements $s_{1},\ldots,s_{m}\in S_{+}$ and $r_{1},\ldots,r_{m}\in R$ such that 
\begin{align*}
r = \sum_{k=1}^{m} s_{k}r_{k}.
\end{align*} 
Note that $s_{k}$ ($1\le k\le m$) are homogeneous elements of $R_{+}$. Thus $\deg(r_{k})=\deg(r)-\deg(s_{k})< \deg(r)=n$. If $n=1$, then $\deg(r_{k})=0$ and $r_{k}\in R_{0}$. Therefore,  $r\in S_{+}R_{0} \subseteq \langle S_{0} \cup S_{+}\rangle$. By induction on $n$, we obtain $r_{k} \in \langle S_{0} \cup S_{+} \rangle$. It follows that
\begin{align*}
r\in S_{+} \cdot \langle S_{0} \cup S_{+} \rangle \subseteq  \langle S_{0} \cup S_{+} \rangle,
\end{align*}
and  $R= \langle S_{0} \cup S_{+} \rangle.$ 

Let $R$ be right noetherian. Then we can choose a finite subset $S_{+}$ of homogeneous elements of $R_{+}$ such that $R_{+}=S_{+}R$. Thus Part (b) follows  from Part (a).
\end{proof}

By Lemma \ref{lem:generators-noetherian-graded-ring}, we have the following result for Noetherian graded rings.
\begin{theorem}\label{thm:noetherian-graded-ring-affine}
Let $R=\oplus_{n\ge 0}R_{n}$ be a graded ring. If $R$ is right or left Noetherian and $R_{0}$ is affine, then $R$ is affine.
\end{theorem}

Next we apply Theorem \ref{thm:noetherian-graded-ring-affine} and Lemma \ref{lem:Molnar} to Noetherian graded algebras whose degree zero part is a commutative or a cocommutative Hopf algebra:
\begin{theorem}\label{thm:noetherian-graded-algebra-affine}
Let $A= \oplus_{n\ge 0}A_{n}$ be a graded algebra over $\K$. Assume that $A$ is noetherian. 
\begin{itemize}
    \item [(a)] If $A_{0}$ is affine, then $A$ is affine.
  \item [(b)] If $A_{0}$ is a commutative or a cocommutative Hopf subalgebra, then $A$ is affine.
\end{itemize}
\begin{proof}
Part (a) follows from Theorem \ref{thm:noetherian-graded-ring-affine}. Note that $A_{0}$ is Noetherian since $A$ is Noetherian and $A_{0}$ is a quotient of $A$.  
By Lemma \ref{lem:Molnar} $A_{0}$ is affine. It follows from Part (a) that $A$ is affine.
\end{proof}
\end{theorem}

It is clear that the above result holds for Hopf algebras which are graded as algebras.
\begin{corollary}\label{cor:noetherian-graded-Hopfalgebra-affine}
Let $H$ be a $\K$-Hopf algebra which is an $\mathbb{N}$-graded algebra. Assume that $H$ is noetherian. Then the following hold:
\begin{itemize}
   \item [(a)] If $H_{0}$ is affine, then $H$ is affine.
  \item [(b)] If $H_{0}$ is a commutative or a cocommutative Hopf subalgebra, then $H$ is affine.
\end{itemize}
\end{corollary}

Denote by $\gr_{c} H$  the associated graded Hopf algebra corresponding to the coradical filtration of a Hopf algebra $H$. We have the following result for pointed Hopf algebras. 
\begin{corollary}\label{cor:noetherian-grH}
Let $H$ be a pointed Hopf algebra over $\K$. If $\gr_{c} H$ is noetherian, then $\corad(H)$, $\gr_{c} H$ and $H$ are affine. 
\end{corollary}
\begin{proof}
Note that $(\gr_{c} H)_0=\corad(H)$ is a cocommutative Noetherian Hopf algebra.  By Lemma \ref{lem:Molnar} it is affine. It follows from Theorem \ref{thm:noetherian-graded-ring-affine} that both $\gr_{c}H$  and $H$ are affine.
\end{proof}

\section{Noetherian graded bialgebras and braided Hopf algebras}

Now we study the affineness of the braided Hopf algebra of a Noetherian graded Hopf algebra. First we recall the following result  about  Radford biproducts of graded Hopf algebras:

\begin{lemma}{\cite{AS1998}}\label{lem:graded-Hopf-algebras}
Let $A=\oplus_{n\ge 0} A_{n}$ be a graded $\K$-Hopf algebra and ${^{A_{0}}_{A_{0}}\mathcal{YD}}$ the category of left Yetter-Drinfel'd modules. Then $A \cong B \sharp A_{0}$ as graded Hopf algebras, where  $B=A^{co~\pi}$ is an $\mathbb{N}$-graded braided Hopf algebra in ${^{A_{0}}_{A_{0}}\mathcal{YD}}$ such that $B=\oplus_{n\ge 0}B_{n}$, $B_{n}:=B\cap A_{n}$, $B_{0}=\K1$ and $B_{+}:= \oplus_{n\ge 1}B_{n} \subseteq A_{+}$.
\end{lemma}

Next we show that the braided Hopf algebra of a Noetherian graded  Hopf algebra is affine.
\begin{theorem}\label{thm:braided-part-affine}
Let $A=\oplus_{n\ge 0} A_{n}$ be a graded Hopf algebra over $\K$. The following hold:
\begin{itemize}
\item  [(a)] Let $G$ be a subset of homogeneous elements of $B_{+}$. Then $B_{+}=GB$ if and only if $A_{+}=GA$, that is, $G$ generates $B_{+}$ as a right ideal of $B$ if and only if $G$ generates $A_{+}$ as a right ideal of $A$. In this case, $B_{+}=GB=\sum_{n\ge 1}\K G^{n}$.
\item  [(b)] If $A$ is right (resp. left) noetherian, then $A_{0}$ is right (resp. left) Noetherian and $B$ is affine.
\end{itemize}
\end{theorem}

\begin{proof}
Suppose that $B_{+}=GB$, that is, $G$ generates $B_{+}$ as a right ideal of $B$. 
By Lemma \ref{lem:graded-Hopf-algebras} and \cite[Theorem 11.7.1]{R2012}, we see that every element $x\in A_{+}$ can be written as 
\begin{align*}
x=\sum_{i} b_{i}a_{i}, \qquad b_{i}\in B_{+},~a_{i}\in A_{0}.
\end{align*} 
Then $x\in GA$, which means that $A_{+}\subseteq GA$. On the other hand,   by Lemma \ref{lem:graded-Hopf-algebras} we have $G\subseteq B_{+} \subseteq A_{+}$. Since $A_{+}$ is a right ideal of $A$, we have  $GA\subseteq A_{+}$. It follows that $A_{+}=GA$. 

Conversely, suppose that $A_{+}=GA$. 
Note that every homogeneous element $y\in A_{+}$ can be written as
\begin{align*}
y= \sum_{j} g_{j}r_{j}, \qquad 0\neq g_{j}\in \K G, ~r_{j}\in A,
\end{align*}
where $r_{j}$ are homogeneous elements of degree $\ge 0$. Since $0\neq g_{j}\in \K G\subseteq A_{+}$, we have $\deg(g_{j})\ge 1$ and $\deg(r_{j})=\deg(y)-\deg(g_{j})<\deg(y)$.  By induction, we  obtain:
\begin{align*}
y = \sum_{k}d_{k}c_{k}, \qquad d_{k}\in   \sum_{n\ge 1}\K G^{n}, ~c_{k}\in A_{0}.
\end{align*}
Recall that $B_{+}\subseteq A_{+}$.
For every homogeneous element $b\in B_{+}$, we can write
\begin{align*}
b= \sum_{k}f_{k}e_{k}, \qquad f_{k}\in  \sum_{n\ge 1}\K G^{n}, ~e_{k}\in A_{0}.
\end{align*}
Now applying the isomorphism  $B \sharp A_{0} \cong A$, we obtain 
$$b\otimes 1_{A_{0}}=\sum_{k}f_{k} \otimes e_{k}.$$ 
It follows that $b \in  \sum_{n\ge 1}\K G^{n}$ and $B_{+}\subseteq  \sum_{n\ge 1}\K G^{n} \subseteq GB$.  Since $G \subseteq B_{+}$ and $B_{+}$ is a right ideal of $B$, we have $B_{+}=GB=\sum_{n\ge 1}\K G^{n}$.

It remains to show Part (b). We first choose a subset $G'\subseteq B_{+}$ such that it generates $B_{+}$ as a right ideal of $B$. 
By Part (a) we have  $A_{+}=G'A$.
Since $A$ is right noetherian, we can choose a finite subset $G\subseteq G'$ such that $A_{+}=GA$. It follows from Part (a)  that  $B_{+}=GB=\sum_{n\ge 1}\K G^{n}$.  So the finite set $G\cup \{1\}$ generates $B$ as an algebra. Therefore,  $B$ is affine.
\end{proof}

From Corollary \ref{cor:noetherian-grH} and Theorem \ref{thm:braided-part-affine} (b), we have the following result for pointed Hopf algebras.
\begin{corollary}
Let $H$ be a pointed Hopf algebra over $\K$. If $\gr_{c} H$ is noetherian, then  $\corad(H)$, $\gr_{c} H$, $H$ and $B$ are affine, where $\gr_{c} H \cong B \sharp \corad(H)$ as grade Hopf algebras.
\end{corollary}

\bibliographystyle{plain}

\begin{thebibliography}{10}

\bibitem{A2023} N. Andruskiewitsch. On infinite-dimensional Hopf algebras. Preprint, 2023, 
arxiv:2308.13120.


\bibitem{AS1998} N. Andruskiewitsch, and H-J. Schneider, Lifting of quantum linear spaces and pointed Hopf algebras of order $p^3$. Journal of Algebra \textbf{209} (1998), 658-691.


\bibitem{B1998}
 K.A. Brown, Representation theory of Noetherian Hopf algebras satisfying a  polynomial identity, in Trends in the Representation Theory of Finite Dimensional Algebras (Seattle 1997), (E.L. Green and B. Huisgen-Zimmermann, eds.), Contemp. Math. \textbf{229} (1998), 49-79. 


\bibitem{B2007} 
K.A. Brown, Noetherian Hopf algebras, Turkish J. Math. \textbf{31} (2007), suppl., 7–23. 


\bibitem{BG2014} K.A. Brown and P. Gilmartin, Hopf algebras under finiteness conditions. Palestine Journal of Mathematics \textbf{3}, 2014.


\bibitem{BZ2020} K.A. Brown and J.J. Zhang. Survey on Hopf algebras of GK-dimension 1 and 2. Preprint, 2020, arXiv:2003.14251.



\bibitem{G2013} K.R. Goodearl, Noetherian Hopf algebras, Glasgow Math. J. \textbf{55} (2013), 75-87; arXiv:1201.4854v1.



\bibitem{M1975}
R.K. Molnar. A commutative Noetherian Hopf algebra over a field is finitely generated, Proc. Amer. Math. Soc. \textbf{51} (1975), 501-502.



\bibitem{R2012}
D.E. Radford.
\newblock {Hopf algebras}, volume \textbf{49}~ of {Series on Knots and Everything}.
\newblock World Scientific Publishing Co. Pte. Ltd., Hackensack, NJ, 2012.


\bibitem{WZ2003} Q.-S. Wu and J.J. Zhang, Noetherian PI Hopf algebras are Gorenstein. Trans. Amer. Math. Soc. \textbf{355} (2003), 1043-1066.
\end{thebibliography}

\end{document}